\documentclass[12pt]{article}
 
 \usepackage{amsmath,amssymb, amsthm}
 
\usepackage{amsmath,amssymb,amsfonts,amsthm}

\usepackage{virginialake}
\usepackage[utf8]{inputenc}

\usepackage{amsmath,amssymb}
\usepackage{bussproofs}
\usepackage{array}
\usepackage{graphicx}

\usepackage{adjustbox} 
\usepackage{float}
\usepackage{tikz}
\usetikzlibrary{positioning,arrows}
\usetikzlibrary{decorations.pathreplacing}
\tikzset{
modal/.style={>=stealth,shorten >=1pt,shorten <=1pt,auto,node distance=1.5cm,semithick},
world/.style={circle,draw,minimum size=0.5cm,fill=gray!15},
point/.style={circle,draw,inner sep=0.5mm,fill=black},
reflexive above/.style={->,loop,looseness=7,in=120,out=60},
reflexive below/.style={->,loop,looseness=7,in=240,out=300},
reflexive left/.style={->,loop,looseness=7,in=150,out=210},
reflexive right/.style={->,loop,looseness=7,in=30,out=330}}

\usepackage{mdframed}

\usepackage{doc}

\usepackage{xcolor}
\usepackage{xspace}
\usepackage{adjustbox}

\usepackage{amsfonts}
\usepackage{graphicx}
\newtheorem{thm}{Theorem}[section]

\newtheorem{cor}[thm]{Corollary}

\newtheorem{lem}[thm]{Lemma}
\newtheorem{dfn}[thm]{Definition}

\newcommand{\fl}{\mbox{$\mathfrak{F}$}}

\newcommand{\AND}{\wedge}
\newcommand{\OR}{\vee}
\newcommand{\IMP}{\rightarrow}

\newcommand{\logicname}[1]{{\sf#1}\xspace}
\newcommand{\f}{\logicname{F}}
\newcommand{\sfour}{\logicname{S4}}

\newcommand{\bpc}{\logicname{BPC}}
\newcommand{\ipl}{\logicname{IPL}}
\newcommand{\klogic}{\logicname{K}}
\newcommand{\wf}{\logicname{WF}}

\newcommand{\propset}{\mathsf{Prop}}

\begin{document}

\title{Sequent calculus for the subintuitionistic logic ${\sf WF_{N_{2}}} $}

\author{\textbf{Fatemeh Shirmohammadzadeh Maleki}\\
Department of Logic, Iranian Institute of Philosophy\\
 Arakelian 4, Vali-e-Asr, Tehran, Iran,~f.shmaleki2012@yahoo.com}

\maketitle

\begin{abstract}
A cut-free G3-style sequent calculus ${\sf GWF_{N_{2}}} $ for the subintuitionistic logic ${\sf WF_{N_{2}}} $, along with its single-succedent variant $ {\sf GWF^{s}_{N_{2}}} $, is introduced. The calculus ${\sf GWF_{N_{2}}} $ is shown to extend naturally to a G3-style of the sequent calculus ${\sf GF}$ for Corsi's logic \f. Additionally, a syntactic proof of the known embedding of ${\sf GWF_{N_{2}}} $ into classical modal logic ${\sf M_{Nec}} $ is presented. 
\end{abstract}


\section{Introduction}

Subintuitionistic logics, first introduced by Corsi~\cite{Co}, study systems weaker than intuitionistic logic. Corsi proposed the basic system \f within a Hilbert-style proof system, characterized by Kripke frames without reflexivity or transitivity and where the preservation of truth is not assumed. He demonstrated a G\"odel-type translation of \f into the modal logic \klogic, analogous to the translation of \ipl into \sfour. Restall~\cite{a2} introduced a similar system, {\sf SJ}.

Visser~\cite{a3} later extended this work by defining Basic Logic (\bpc) in natural deduction form and proving its completeness for finite, irreflexive Kripke models. Subsequent studies by Ardeshir and Ruitenburg further explored \bpc (see, e.g.,~\cite{Ar0, Ar1, Ar}).

Building on these developments, de Jongh and Shirmohammadzadeh Maleki~\cite{Dic3, Di, FD} introduced weaker subintuitionistic logics based on neighborhood semantics, diverging from traditional Kripkean approaches. Their work emphasized the system \wf, which is significantly weaker than \f, and established correspondences with modal logic through modal companions~\cite{Di, FD6, Dic5}. For example, they proved that the classical modal logic ${\sf M_{Nec}} $ is the modal companion of the subintuitionistic logic ${\sf WF_{N_{2}}} $, which is obtained by adding a rule to \wf and is weaker than the subintuitionistic logic \f~\cite{Di}.

The development of sequent calculi for subintuitionistic logics has been a prominent area of research. Notable contributions include a dual-context sequent calculus for \f by Kikuchi~\cite{kik}, a cut-elimination proof for \f by Ishigaki and Kashima~\cite{ish}, and a labeled sequent calculus by Yamasaki and Sano~\cite{yam}. Aboolian and Alizadeh~\cite{ab} proposed a G3-style sequent calculus for \f to establish Lyndon's and Craig's interpolation properties. Additionally, Tesi~\cite{Tesi} introduced nested calculi for subintuitionistic logics corresponding to modal logics such as {\sf K}, {\sf T}, and {\sf K4}. In~\cite{Fatemeh}, Shirmohammadzadeh Maleki extended the study of subintuitionistic logics by introducing sequent systems for \wf, \f, and intermediate logics between these systems.

In this paper, we extend the study of subintuitionistic logics by introducing new sequent system for the subintuitionistic logic ${\sf WF_{N_{2}}} $. We investigate the admissibility of structural rules, such as weakening, contraction, and cut, and establish equivalences between Hilbert-style and sequent systems. Furthermore, a simple syntactic proof of embedding result of ${\sf WF_{N_{2}}} $ into ${\sf M_{Nec} }$ is provided.

The paper is organized as follows:
Section~\ref{sec:modneigh} reviews axiomatic systems and neighborhood semantics for the subintuitionistic logic ${\sf WF_{N_{2}}}$.
Section~\ref{modneigh1} introduces the sequent calculus for ${\sf WF_{N_{2}}}$, called ${\sf GWF_{N_{2}}}$, demonstrating the admissibility of weakening, contraction, and cut. Moreover, we establish the equivalence between the Hilbert-style proof system and the corresponding sequent calculus system.
In Section~\ref{embedding}, we demonstrate the syntactic embedding of ${\sf WF_{N_{2}}}$ into classical modal logic ${\sf M_{Nec}}$ as a result of the consequence relation.
In Section~\ref{some}, a single-succedent version of ${\sf GWF_{N_{2}}}$ is presented. It is shown to have the same strength as ${\sf GWF_{N_{2}}}$ when applied to formulas of ${\sf WF_{N_{2}}}$.

\section{The subintuitionistic logic ${\sf WF_{N_{2}}} $}\label{sec:modneigh}
In this section, we provide the technical background related to the subintuitionistic logic ${\sf WF_{N_{2}}} $. The results here have been proved before in \cite{Di}.

The language of subintuitionistic logics consists of a countable set of atomic propositions, represented by lowercase letters \( p, q, \dots \) from the set \( \propset \), combined using the connectives \( \vee, \wedge, \rightarrow \), and the propositional constant \( \bot \). Formulas in this language are denoted by uppercase Latin letters \( A, B, C, \dots \). The biconditional \( \leftrightarrow \) is introduced as a shorthand, defined by 
$A \leftrightarrow B \equiv (A \rightarrow B) \wedge (B \rightarrow A)$.

 \begin{dfn}\label{Nframe}
An  \textit{N-neighbourhood Frame} $  \fl\,{=}\,\langle W, N\rangle $  for subintuitionistic logic consists of a non-empty set $W$, and a function
$N$  from $W$ into $ \mathcal{P}(\mathcal{P}(W))$ such that for each $ w \in W,\,W \in N(w) $.
The N-neighborhood frame $  \fl\,{=}\,\langle W, N\rangle $  is closed under
\textit{superset} if and only if for all $ w \in W$, if $X \in N(w)$  and $X\subseteq Y  $,
then $Y \in N(w)$.

\noindent Valuation $ V\!:At\rightarrow \mathcal{P}(W) $ makes $\mathfrak{M}= \langle W, N, V\rangle $ an \textit{\emph{N}-neighborhood Model}. 
Truth of a propositional formula in a world $w$ is defined inductively as follows.
\begin{enumerate}
\item $ \mathfrak{M},w \Vdash p~~~~~~ ~~\Leftrightarrow~~ w \in V(p)$;
\item $ \mathfrak{M},w \Vdash A\wedge B~~\Leftrightarrow~~ \mathfrak{M},w \Vdash A ~{\rm and}~ \mathfrak{M},w \Vdash B$;
\item $ \mathfrak{M},w \Vdash A\vee B~~\Leftrightarrow ~~\mathfrak{M},w \Vdash A ~{\rm or}~ \mathfrak{M},w \Vdash B$;
\item $ \mathfrak{M},w \Vdash A\rightarrow B \,\,\,\Leftrightarrow\,\,\,   \left\lbrace v~|~v\Vdash A ~\Rightarrow~v\Vdash B\right\rbrace =\overline{A^{{\mathfrak M}}}\cup B^{{\mathfrak M}} \in N(w)$;
\item $ \mathfrak{M},w \nVdash \perp,$
\end{enumerate}
where $ A^{{\mathfrak M}}:=\left\lbrace  w \in W~|~\mathfrak{M},w\Vdash A\right\rbrace  $. A formula $A$ is \textit{valid} in $\mathfrak{M}$, $\mathfrak{M}\,\,{\Vdash}\, A$, if for all $w \in W, \,\mathfrak{M}, w\Vdash A$, and $ A $ is valid in $ \fl $, $ \fl \,\,{\Vdash}\, A$ if for all $ \mathfrak{M} $ on $ \fl $, $ \mathfrak{M} \Vdash A $.  We write $\,{\Vdash}\, A$ if $\mathfrak{M}\,{\Vdash}\, A$ for all $\,\mathfrak{M}$. Also we define $ \Gamma\,{\Vdash}\, A $ iff for all $ \mathfrak{M} $,$ w \,{\in}\, \mathfrak{M} $, if $ \mathfrak{M} , w \Vdash \Gamma$ then $  \mathfrak{M} , w \Vdash A$.
\end{dfn}

\begin{table}[h]
\caption{The Hilbert-style system for ${\sf WF_{N_{2}}} $ }\label{fig:axioms:wf}
\begin{center}
\scalebox{0.85}{
\begin{tabular}{cc}
\hline\hline
\noalign{\smallskip}
& \\[5pt]
 1. $ A \IMP (A \OR B )$ 
& 
2. $ B \IMP (A \OR B )$
\\[5pt]

 3. $ (A \AND B) \IMP A$
&
4. $ (A \AND B) \IMP B$
\\[5pt]

5. $ A \AND(B \OR C) \IMP (A \AND B ) \OR (A \AND C) $
&
6. $A \IMP A$ 
\\[5pt]

7. $\vliinf{}{}{B}{A}{A \IMP B}$ 
&
8. $\vlinf{}{}{B \IMP A}{A} $
\\[5pt]

9. $\vliinf{}{}{A\IMP C}{A \IMP B}{B \IMP C}$
&
10. $\vliinf{}{}{A \IMP (B \AND C)}{A \IMP B }{A \IMP C}$
\\[5pt]

11.  $\vliinf{}{}{(A \OR B ) \IMP C}{A \IMP C}{B \IMP C}$
&
12.  $\vliinf{}{}{ A \AND B}{A}{B}$
\\[5pt]

13. $\vliinf{}{}{(A \IMP B) \IMP (C \IMP D) }{C\IMP A\OR D}{C\AND B\IMP D}$
&
14. $\bot \IMP A$
\\[5pt]
\noalign{\smallskip}\hline\hline
\end{tabular}}
\end{center}
\end{table}

 \begin{dfn}\label{def:hilbert:wf} 
The Hilbert-style axiomatization of the subintuitionistic logic ${\sf WF_{N_{2}}} $ consists of the axioms and inference rules reported in Table~\ref{fig:axioms:wf}. 
\end{dfn}  
In Table~\ref{fig:axioms:wf}, the rules should be applied in such a way that, if the formulas above the line are theorems of ${\sf WF_{N_{2}}}$, then the formula below the line is a theorem as well and we will refer to rules 7, 8, 12 and 13 as the modus ponens (MP), a fortiori (AF), conjunction and (${\sf  N_{2}} $) rules, respectively.
The basic notion $ \vdash_{{\sf WF_{N_{2}}} } A$ means that $ A $ can be drived from the axioms of the ${\sf WF_{N_{2}}}$  by means of its rules. But, when one axiomatizes local validity, not all rules in a Hilbert type system have the same status when one considers deductions from assumptions.
 In the case of deductions from assumptions, we impose restrictions on all rules except the conjunction rule.
The restriction on modus ponens is slightly weaker than on the other rules.
The following definition arises from these considerations.

 \begin{dfn}\label{hh}
We define $ \Gamma\,{\vdash_{{\sf WF_{N_{2}}}}} A $ iff there is a  derivation of $A$ from $ \Gamma $ using the rules 8, 9, 10 , 11 and 13 of Table~\ref{fig:axioms:wf}, only when there are  no assumptions, and the rule 5, MP, only when the derivation of $ A\rightarrow B $ contains no assumptions.
\end{dfn}  
 
By the definition of $ \Gamma\,{\vdash_{{\sf WF_{N_{2}}}}} A $, a weak form of the deduction theorem with a single assumption is obtained:

\begin{thm}\label{z3}
{\rm (Weak Deduction theorem, \cite{Di})}
\begin{enumerate}
\item[(a)] $A\vdash_{{\sf WF_{N_{2}}}} B  ~$ iff $ ~\vdash_{{\sf WF_{N_{2}}}} A\rightarrow B $.
\item[(b)] $A_1,\dots,A_n\vdash_{{\sf WF_{N_{2}}}} B  ~$ iff $~ \vdash_{{\sf WF_{N_{2}}}} A_1\wedge\dots\wedge A_n\rightarrow B $.
\end{enumerate}
\end{thm}

\begin{thm} (\cite{Di}, Theorem 14)
The  subintuitionistic logic ${\sf WF_{N_{2}}} $ is sound and strongly complete with respect to
the class of N-neighborhood frames that are closed under superset.
\end{thm}

\section{ The system ${\sf  GWF_{N_{2}}} $}\label{modneigh1}

\begin{table}[h]
\caption{The sequent calculus ${\sf  GWF_{N_{2}}} $}\label{wfn2}
\begin{center}
\scalebox{0.85}{
\begin{tabular}{cc}
\hline\hline
\noalign{\smallskip}
\textbf{Initial sequents:} & \\[5pt]

\AxiomC{}
\RightLabel{\(id\)}
\UnaryInfC{$p, \Gamma \Rightarrow \Delta, p$}
\DisplayProof
&
\AxiomC{}
\RightLabel{\(L_{\bot}\)}
\UnaryInfC{$\bot, \Gamma \Rightarrow \Delta$}
\DisplayProof\\[10pt]

\multicolumn{2}{l}{ \( p \) is an atomic variable.} \\[10pt]
\noalign{\smallskip}\hline\noalign{\smallskip}
\textbf{logical rules:} & \\[5pt]
\AxiomC{$X, Y, \Gamma \Rightarrow \Delta$}
\RightLabel{\(L_{\wedge}\)}
\UnaryInfC{$X \wedge Y, \Gamma \Rightarrow \Delta$}
\DisplayProof
& 
\AxiomC{$\Gamma \Rightarrow \Delta, X$}
\AxiomC{$\Gamma \Rightarrow \Delta, Y$}
\RightLabel{\(R_{\wedge}\)}
\BinaryInfC{$\Gamma \Rightarrow \Delta, X \wedge Y$}
\DisplayProof \\[10pt]

\AxiomC{$X, \Gamma \Rightarrow \Delta$}
\AxiomC{$Y, \Gamma \Rightarrow \Delta$}
\RightLabel{\(L_{\lor}\)}
\BinaryInfC{$X \lor Y, \Gamma \Rightarrow \Delta$}
\DisplayProof
&
\AxiomC{$\Gamma \Rightarrow \Delta, X, Y$}
\RightLabel{\(R_{\vee}\)}
\UnaryInfC{$\Gamma \Rightarrow \Delta, X \vee Y$}
\DisplayProof\\[10pt]

\AxiomC{$\Gamma \Rightarrow \Delta, A$}
\AxiomC{$B, \Gamma \Rightarrow \Delta$}
\RightLabel{\(L_{\supset}\)}
\BinaryInfC{$A \supset B, \Gamma \Rightarrow \Delta$}
\DisplayProof
&
\AxiomC{$A, \Gamma \Rightarrow \Delta, B$}
\RightLabel{\(R_{\supset}\)}
\UnaryInfC{$\Gamma \Rightarrow \Delta, A \supset B$}
\DisplayProof \\[10pt]

\AxiomC{$C\supset D, A \Rightarrow  B$}
\RightLabel{\(LR_{\to} \)}
\UnaryInfC{$ \Gamma , C\rightarrow D\Rightarrow \Delta, A \rightarrow B$}
\DisplayProof
&
\AxiomC{$A \Rightarrow  B$}
\RightLabel{\(R_{\to} \)}
\UnaryInfC{$\Gamma \Rightarrow \Delta, A \rightarrow B$}
\DisplayProof\\[14pt]

\multicolumn{2}{l}{ \( A \), \( B \), \( C \), \( D \) \( \in \) \( \sf Frm \) and \( X \), \( Y \) \( \in \)\( \sf Frm_2 \).  } \\[10pt]

\noalign{\smallskip}\hline\hline
\end{tabular}}
\end{center}
\end{table}

In this section, we present the sequent calculus system $ {\sf  GWF_{N_{2}}} $. To formally represent strict implication $ \rightarrow $ in the logic $ {\sf  GWF_{N_{2}}} $, inspired by \cite{sa}, we introduce a novel expression 
$ A\supset B $, which is intended to signify the material implication found in classical logic.

Let $ \mathcal{L}=\lbrace  \wedge, \vee, \rightarrow, \bot\rbrace $. We denote the set of all formulas constructed from $ \mathcal{L}  $ as ${\sf Frm}$, with uppercase Latin letters $ A, B, C, \cdots $ serving as metavariables for these formulas. Additionally, we define a new set $ { \sf Frm_1}:={\sf Frm}\cup\lbrace A\supset B ~|~A, B \in {\sf Frm}\rbrace$. Following the approach in \cite{ab}, we construct $ { \sf Frm_2}$ as the set of expressions formed by the conjunctions and disjunctions of  $ {\sf Frm_{1}} $-formulas. It is important to note that nested material implications $ \supset $ are not permitted in this system.

For simplicity and clarity, we also refer to members of ${ \sf Frm_1} $ and ${ \sf Frm_2} $ as formulas and use the symbols $ X, Y, Z, \cdots $ as metavariables for both sets. A sequent in this system takes the form 
$ \Gamma\Rightarrow\Delta $, where $ \Gamma $ and $ \Delta $ are finite multisets of formulas drawn from ${ \sf Frm_2}$.

Note that in the rules, the multisets $ \Gamma $ and $  \Delta$ are called \textit{side formulas}, the formula in
the conclusion is called \textit{principal}, and the formulas in the premises are called \textit{active}.
We call a formula $  A$ \textit{strict implicational}, if $ A=B\rightarrow C $ for some $  B, C \in  {\sf Frm}$.

\begin{dfn}
The rules of the calculus ${\sf  GWF_{N_{2}}}$ for the basic subintuitionistic logic ${\sf  WF_{N_{2}}}$ reported in Table \ref{wfn2}.
\end{dfn}
The concept of a proof tree is defined as usual. For a proof tree $ \mathcal{D} $ ending with the sequent 
$ \Gamma\Rightarrow\Delta $, the height of 
$ \mathcal{D} $ is the maximum number of successive rule applications, with initial sequents having a height of $ 0 $.
By, $ \vdash_{n} \Gamma\Rightarrow\Delta $ we indicate that the sequent 
$ \Gamma\Rightarrow\Delta $ is derivable in ${\sf  GWF_{N_{2}}}$  with a proof of height at most $ n $.

It is worth mentioning that if the rule \( (LR\rightarrow) \) is generalized to the rule  
\[
\dfrac{\Pi^{\supset}, A \Rightarrow B}{\Pi, \Gamma \Rightarrow \Delta, A \rightarrow B}
\]
in which \( \Pi \) is a multiset consisting of strict implicational formulas and \( \Pi^{\supset} \) is obtained by replacing each formula \( A \rightarrow B \in \Pi \) with \( A \supset B \), then the system \( {\sf GF} \), introduced by Aboolian and Alizadeh \cite{ab}, is obtained. Furthermore, if the condition on \( \Pi \) in this rule is modified so that \( \Pi \) becomes an arbitrary multiset of \( {\sf Frm_{1}} \) formulas, the system \( {\sf GBPC3} \) for Visser's basic propositional logic, introduced by Aghaei and Ardeshir \cite{ar}, is derived.
\subsection{Admissibility of structural rules}

In this section we will prove the admissibility of structural rules for the calculus  ${\sf  GWF_{N_{2}}}$.
\begin{lem}\label{1}
The sequent $ X, \Gamma\Rightarrow \Delta, X $ is derivable in ${\sf  GWF_{N_{2}}}$  for an arbitrary formula $ X $ and arbitrary side formulas $ \Gamma $ and $\Delta$.
\end{lem}
\begin{proof}
The proof is similar to the proof of Lemma 2.2 in \cite{ab}.
\end{proof}

The next property we aim to establish is the invertibility of rules. A rule is said to be height-preserving invertible \textit{(hp-invertible)} if, whenever the conclusion of a rule instance is derived, there exists a derivation of its premise(s) with a height that is less than or equal to the height of the derivation of the conclusion.
\begin{lem}\label{2}
\textbf{(Inversion Lemma)} Every rule in ${\sf  GWF_{N_{2}}}$, with the exception of \(LR\to \) and \(R\to \), is hp-invertible.
\end{lem}
\begin{proof}
The proof is similar to the proof of Lemma 2.3 in \cite{ab}.
\end{proof}

\begin{lem}\label{3}
If $ \vdash_{n} \Rightarrow A\rightarrow B $, then $ \vdash_{n}  A\Rightarrow B $.
\end{lem}
\begin{proof}
The proof is similar to the proof of Lemma 2.4 in \cite{ab}.
\end{proof}
A rule of inference is said to be height-preserving admissible \textit{(hp-admissible)} in ${\sf  GWF_{N_{2}}}$ if, whenever its premises are derivable in ${\sf  GWF_{N_{2}}}$, then also its conclusion is derivable (with at most the same derivation height) in ${\sf  GWF_{N_{2}}}$.

\begin{lem}\label{1b}
The left and right rules of weakening are hp-admissible in ${\sf  GWF_{N_{2}}}$.
$$ \frac{\Gamma\Rightarrow \Delta}{D, \Gamma\Rightarrow \Delta}L_{w}~~~~~~~~~~~~~\frac{\Gamma\Rightarrow \Delta}{\Gamma\Rightarrow \Delta, D}R_{w}$$
\end{lem}
\begin{proof}
The proof is similar to the proof of Lemma 2.6 in \cite{ab}.
\end{proof}

\begin{lem}\label{1c}
The left and right rules of contraction are hp-admissible in ${\sf  GWF_{N_{2}}}$.
$$\frac{D, D, \Gamma\Rightarrow \Delta}{D, \Gamma\Rightarrow \Delta}L_{c}~~~~~~~~\frac{ \Gamma\Rightarrow \Delta, D, D}{\Gamma\Rightarrow \Delta, D}R_{c}$$
\end{lem}
\begin{proof}
Assume $ \vdash_{n} D, D, \Gamma\Rightarrow \Delta$ and  $ \vdash_{n} \Gamma\Rightarrow \Delta, D, D$. We need to show that $ \vdash_{n} D, \Gamma\Rightarrow \Delta$ and  $ \vdash_{n} \Gamma\Rightarrow \Delta, D$. The proof is by  induction on $ n $. The base case ($ n =0$) is straightforward. Suppose the statement is true for $ n $ and we have $ \vdash_{n+1} D, D, \Gamma\Rightarrow \Delta$ and  $ \vdash_{n+1} \Gamma\Rightarrow \Delta, D, D$. We just prove the cases when the last rule is $(R_{\rightarrow})$ or $(LR_{\rightarrow})$. For the proof of other cases you can refer to the proof of  Lemma 2.7 in \cite{ab}.

Suppose  $ \vdash_{n+1} D, D, \Gamma\Rightarrow \Delta$ and  the last rule is $(R_{\rightarrow})$. Then $ D, D, \Gamma=\Sigma $ and $ \Delta = \Delta^{'}, A\rightarrow B $. So $ \vdash_{n} A\Rightarrow B$. By $(R_{\rightarrow})$ we have $ \vdash_{n+1} D, \Gamma\Rightarrow \Delta$.

Suppose  $ \vdash_{n+1} D, D, \Gamma\Rightarrow \Delta$ and  the last rule is $(LR_{\rightarrow})$. Then $ D, D, \Gamma=\Sigma , C\rightarrow D $ and $ \Delta = \Delta^{'}, A\rightarrow B $. There are two cases as follows:

$\ast $ $ D= C\rightarrow E $ and $ \Sigma= C\rightarrow E, \Gamma $. Then $ D, D, \Gamma=C\rightarrow E, \Sigma $ and $ \Delta = \Delta^{'}, A\rightarrow B $. So  $ \vdash_{n} C\supset E, A\Rightarrow B$. By $(LR_{\rightarrow})$ we have $ \vdash_{n+1} D, \Gamma\Rightarrow \Delta$.

$\ast $  $\Gamma=\Gamma^{'}, C\rightarrow E $ and $ \Sigma=D, D, \Gamma^{'} $. So  $ \vdash_{n} C\supset E, A\Rightarrow B$. By $(LR_{\rightarrow})$ we have $ \vdash_{n+1} D, \Gamma\Rightarrow \Delta$.

Suppose  $ \vdash_{n+1} \Gamma\Rightarrow D, D, \Delta$ and  the last rule is $(LR_{\rightarrow}) $. Then $ \Gamma=\Pi, C\rightarrow E$ and $ \Delta, D, D = \Delta^{'}, A\rightarrow B $. There are two cases as follows:

$\ast $ $ D=A\rightarrow B $ and $ \Delta, D=\Delta^{'} $. So $\vdash_{n}  \Pi^{\supset}, A\Rightarrow B$. By $(LR_{\rightarrow}) $ we have $ \vdash_{n+1} \Pi, \Sigma \Rightarrow A\rightarrow B, \Delta$.

$\ast $ $ \Delta=A\rightarrow B, \Delta^{''} $ and $ \Delta^{''}, D, D=\Delta^{'} $. So $\vdash_{n}  \Pi^{\supset}, A\Rightarrow B$. By $(LR_{\rightarrow}) $ we have $ \vdash_{n+1} \Pi, \Sigma \Rightarrow A\rightarrow B, \Delta^{''}$.

\end{proof}

In the following we define weight of formulas and cut-height of derivations.

\begin{dfn}
The weight of a formula $  A$ in  ${\sf Frm_{2}}$ is defined inductively as follows:
\begin{enumerate}
\item $ w(\bot)=w(p)=0 $~~~~~~~~~~~~~~~~~~~~~~~~~~for atoms $ p $
\item $ w(A\rightarrow B)=w(A)+w(B)+2 $~~~~~~~~$ A\rightarrow B \in {\sf Frm} $
\item $ w(A\supset B)=w(A)+w(B)+1 $~~~~~~~ $ A\supset B \in {\sf Frm_{1}} $
\item $ w(A\circ B)=w(A)+w(B)+1 $~~~~~~~~~$ \circ \in \lbrace \wedge, \vee\rbrace $
\end{enumerate}
\end{dfn}
\begin{dfn}
 The \textit{cut-height}  of an instance of the rule of cut in a derivation is the sum of heights of derivation of the two premises of cut. 
\end{dfn}
\begin{thm}\label{1d}
The rule of cut  is addmissible in ${\sf  GWF_{N_{2}}}$.
$$ \frac{\Gamma\Rightarrow  D, \Delta~~~~~~D, \Gamma^{'}\Rightarrow \Delta^{'}}{\Gamma,\Gamma^{'} \Rightarrow \Delta, \Delta^{'}}Cut$$
\end{thm}
\begin{proof}
The proof is by induction on the weight of the cut formula $ D $ with a sub-induction on the cut-height. Our strategy is to successively remove topmost cuts.
The proof follows exactly the same steps as the proof of Theorem 2.10 in \cite{ab}. The only difference arises in the case where cut formula $  D$ is principal in both premises and $  D=E\rightarrow F$, which is addressed in detail below. We transform

$$ \frac{\frac{\Pi^{\supset}, E\Rightarrow F}{\Pi, \Sigma\Rightarrow E\rightarrow F, \Delta}{LR_{\rightarrow}}~~~~~\frac{E\supset F, A\Rightarrow B}{E\rightarrow F, \Sigma^{'}\Rightarrow \Delta^{'}, A\rightarrow B}LR_{\rightarrow}}{\Pi, \Sigma, \Sigma^{'}\Rightarrow \Delta, \Delta^{'}, A\rightarrow B}Cut $$
\begin{center}
into
\end{center}
$$ \frac{\frac{\frac{\Pi^{\supset}, E\Rightarrow F}{\Pi^{\supset}\Rightarrow E\supset F} R_{\supset}~~~~~\genfrac{}{}{0pt}{}{}{E\supset F, A\Rightarrow B}}{\Pi^{\supset}, A\Rightarrow B}Cut}{\Pi, \Sigma, \Sigma^{'}\Rightarrow \Delta, \Delta^{'}, A\rightarrow B} LR_{\rightarrow}$$
\end{proof}

We now proceed to define the concept of an extended subformula for formulas in ${\sf  Frm_{2}} $.

\begin{dfn}
Let $ sub(A) $ denote the set of all subformulas of $ A $ for $ A \in {\sf Frm }$. The set $ sub^{\ast}(A) $ for $ X \in {\sf Frm_{2}} $ is called the set of extended subformulas of $ X $ and  defined as follows:
\begin{enumerate}
\item $ sub^{\ast}(p)=\lbrace p\rbrace $~~~~~for atoms $ p $,
\item $ sub^{\ast}(\perp)=\lbrace \perp\rbrace $,
\item $ sub^{\ast}(X \circ Y)=\lbrace X \circ Y\rbrace \cup sub^{\ast}(X ) \cup sub^{\ast}(Y)$~~~~for $ \circ \in \lbrace\wedge, \vee, \supset\rbrace $,
\item $ sub^{\ast}(A\rightarrow B)=\lbrace A\rightarrow B\rbrace \cup sub^{\ast}(A\supset B )$.
\end{enumerate}
\end{dfn}

\begin{cor}
The system  ${\sf  GWF_{N_{2}}}$ has the extended subformula property. That is, every formula occurring in a proof tree appears in $sub^{\ast}(Y)$ for some $ Y $ present in the root of the proof tree.
\end{cor}

\begin{lem}
Let $ \vdash\Gamma\Rightarrow\Delta $. If $ \Delta $ is a multisubset of $ {\sf Frm }$, then the rule $(R_{\supset} )$ is not used in the derivation of $ \Gamma\Rightarrow\Delta  $.
\end{lem}
\begin{proof}
The proof follows exactly the same steps as the proof of Lemma 2.14 in \cite{ab}.
\end{proof}

\subsection{The relation between ${\sf  GWF_{N_{2}}}$ and 
${\sf  WF_{N_{2}}}$}
In this section, we will demonstrate that, with respect to formulas in ${\sf Frm } $, 
${\sf  GWF_{N_{2}}}$ and 
${\sf  WF_{N_{2}}}$ possess the same expressive power.

\begin{thm}\label{3a}
Let $ \Gamma $ and $  \Delta$ be finite subsets of $ {\sf Frm }$. If $ \vdash_{{\sf  WF_{N_{2}}}} \bigwedge \Gamma \rightarrow \bigvee \Delta  $ then $ {\sf  GWF_{N_{2}}}  \vdash \Gamma \Rightarrow \Delta  $.
\end{thm}
\begin{proof}
The proof proceeds by induction on the height of the axiomatic derivation in ${\sf  WF_{N_{2}}}$. In order to prove the base case, we just need to show that all axioms  of  ${\sf  WF_{N_{2}}}$ can be deduced
in ${\sf  GWF_{N_{2}}}$. A straightforward application of the rules of the sequent calculus ${\sf  GWF_{N_{2}}}$, possibly using Lemma \ref{1}. The proof of all axioms are easy.

For the inductive steps, we need to show that the rules of ${\sf  WF_{N_{2}}}$, considering the Definition~\ref{hh}, can be deduced
in ${\sf  GWF_{N_{2}}}$. As an example, if the last step is by the rule $ {\sf N_{2}}$, then $ \Gamma =\emptyset $ and $ \Delta $ is $ (A\rightarrow B)\rightarrow (C\rightarrow D) $. We know that we have derived $ (A\rightarrow B)\rightarrow (C\rightarrow D) $ from  $ C\rightarrow A\vee D $, and $ C\wedge B\rightarrow D $. Thus we assume, by induction hypothesis, that 
$ {\sf  GWF_{N_{2} }  } \vdash \Rightarrow  C\rightarrow A\vee D $  and $ {\sf  GWF_{N }  }  \vdash \Rightarrow  C\wedge B\rightarrow D $. By Lemmas \ref{2} and \ref{3}, we obtain that  
$ {\sf  GWF_{N_{2} }  } \vdash  C\Rightarrow  A, D $ and $ {\sf  GWF_{N_{2} }  }  \vdash   C, B\Rightarrow D $. We can thus proceed as follows:
$$ \frac{\frac{\frac{\frac{}{C\Rightarrow A, D}\!I\!H+\ref{2}+\ref{3}~~~~~~\frac{}{C, B\Rightarrow D}\!I\!H+\ref{2}+\ref{3}}{A\supset B, C\Rightarrow D}L_{\supset}}{A\rightarrow B\Rightarrow C\rightarrow D}LR_{\rightarrow}}{\Rightarrow (A\rightarrow B)\rightarrow (C\rightarrow D)}LR_{\rightarrow} $$
The proof of other cases are easy.
\end{proof}

\begin{thm}\label{3a}
Let $ \Gamma $ and $  \Delta$ be finite subsets of $ {\sf Frm} $. If  $ {\sf  GWF_{N_{2}}}  \vdash \Gamma \Rightarrow \Delta  $ then $ \vdash_{{\sf  WF_{N_{2}}}} \bigwedge \Gamma \rightarrow \bigvee \Delta  $.
\end{thm}
\begin{proof}
The proof proceeds by induction on the height of the derivation in {\sf  GWF}.
If the derivation has height $  0$, we have an axiom or an instance on $ \bot_{L} $. In both cases the claim holds. If the height is $n + 1$, we consider the last rule applied in the derivation.

We just prove for the rule $(LR_{\rightarrow} )$  as follows. The proof of other rules are similar and easy. Note that $( L_{\supset} )$ or $ (R_{\supset}) $ is not applicable, as their principal formulas lie outside $ {\sf  Frm} $.

Assume the last derivation is by the rule $(LR_{\rightarrow} )$.
In this case we have ${\sf  GWF_{N_{2}} }  \vdash_{n} \Gamma^{'} , C\rightarrow D \Rightarrow A\rightarrow B, \Delta^{'}$. We need to show that  $\vdash_{\sf WF_{N_{2}}}   \Gamma^{'} \bigwedge (C\rightarrow D) \rightarrow (A\rightarrow B)\bigvee \Delta^{'} $. By assumption we have $ \vdash_{n}C\supset D , A\Rightarrow B$.
Hence we have the following:
\begin{enumerate}
\item $ \vdash_{{\sf  WF_{N_{2}}}}  A\rightarrow C\vee B $~~~~~~~~~~~~~~~~by IH
\item$\vdash_{\sf WF_{N_{2}}}   D\wedge A\rightarrow B$~~~~~~~~~~~~~~~~by IH
\item$\vdash_{\sf WF_{N_{2}}} (C\rightarrow D)\rightarrow (A\rightarrow B)$~~~~~~~~~~~by 1, 2 and rule ${\sf N_{2}}$
\item$\vdash_{\sf WF_{N_{2}}} \Gamma^{'} \bigwedge (C\rightarrow D)\rightarrow (A\rightarrow B)\bigvee \Delta^{'}$
\end{enumerate}
\end{proof}

\section{Syntactic embedding of ${\sf  WF_{N_{2}}}$ in ${\sf M_{Nec}}$}\label{embedding}

 \begin{table}[h]
\caption{The sequent calculus ${\sf G3{\sf M_{Nec}} }$}\label{G3MN}
\begin{center}
\scalebox{0.85}{
\begin{tabular}{cc}
\hline\hline
\noalign{\smallskip}
\textbf{Initial sequents:} & $\qquad p, \Gamma \Rightarrow \Delta, p \qquad p$ propositional variable \\
\noalign{\smallskip}\hline\noalign{\smallskip}
\textbf{Propositional rules:} & \\[5pt]
\AxiomC{$A, B, \Gamma \Rightarrow \Delta$}
\RightLabel{\(L_{\wedge_{\square}}\)}
\UnaryInfC{$A \wedge B, \Gamma \Rightarrow \Delta$}
\DisplayProof
& 
\AxiomC{$\Gamma \Rightarrow \Delta, A$}
\AxiomC{$\Gamma \Rightarrow \Delta, B$}
\RightLabel{\(R_{\wedge_{\square}}\)}
\BinaryInfC{$\Gamma \Rightarrow \Delta, A \wedge B$}
\DisplayProof \\[10pt]

\AxiomC{$A, \Gamma \Rightarrow \Delta$}
\AxiomC{$B, \Gamma \Rightarrow \Delta$}
\RightLabel{\(L_{\lor_{\square}}\)}
\BinaryInfC{$A \lor B, \Gamma \Rightarrow \Delta$}
\DisplayProof
&
\AxiomC{$\Gamma \Rightarrow \Delta, A$}
\AxiomC{$\Gamma \Rightarrow \Delta, B$}
\RightLabel{\(R_{\lor_{\square}}\)}
\BinaryInfC{$\Gamma \Rightarrow \Delta, A \lor B$}
\DisplayProof\\[10pt]

\AxiomC{$\Gamma \Rightarrow \Delta, A$}
\AxiomC{$B, \Gamma \Rightarrow \Delta$}
\RightLabel{\(L_{\supset_{\square}}\)}
\BinaryInfC{$A \supset B, \Gamma \Rightarrow \Delta$}
\DisplayProof
&
\AxiomC{$A, \Gamma \Rightarrow \Delta, B$}
\RightLabel{\(R_{\supset_{\square}}\)}
\UnaryInfC{$\Gamma \Rightarrow \Delta, A \supset B$}
\DisplayProof \\[10pt]

\AxiomC{}
\RightLabel{\(L_{\bot_{\square}}\)}
\UnaryInfC{$\bot, \Gamma \Rightarrow \Delta$}
\DisplayProof \\[5pt]
\noalign{\smallskip}\hline\noalign{\smallskip}
\textbf{Modal rules:} & \\[5pt]
\AxiomC{$A \Rightarrow B$}
\RightLabel{\(LR_{M_{\square}}\)}
\UnaryInfC{$\square A, \Gamma \Rightarrow \Delta, \square B$}
\DisplayProof 
&
\AxiomC{$ \Rightarrow B$}
\RightLabel{\(R_{N_{\square}}\)}
\UnaryInfC{$ \Gamma \Rightarrow \Delta, \square B$}
\DisplayProof \\[5pt]
\hline\hline
\end{tabular}}
\end{center}
\end{table}
De Jongh and Shirmohammadzadeh Maleki, in \cite{Di}, demonstrated via a semantic approach that the classical modal logic $\mathsf{M_{Nec}}$ serves as the modal companion of the subintuitionistic logic $\mathsf{WF_{N_2}}$. In this section, we provide a syntactic proof of this result.

We consider the language of modal propositional logic, $ \mathcal{L}_{\square} $, defined by the set $ \lbrace \wedge, \vee, \supset, \bot, \square \rbrace $, and the symbol $ \leftrightarrow $ is defined as $ A \leftrightarrow B = (A \supset B) \wedge (B \supset A) $. A system of modal logic is \textit{classical} if and only if it is closed under the rule $ RE $ \emph{($ \frac{A \leftrightarrow B}{\square A \leftrightarrow \square B} $)}. The logic {\sf E} is the smallest classical modal logic. Completeness for {\sf E} is established with respect to modal neighborhood frames~\cite{Che}.

A system of modal logic is \textit{monotonic} if and only if  it is closed under RM \emph{($ \frac{A\rightarrow B}{\square A\rightarrow \square B} $)}. {\sf M} ({\sf EM}) is the smallest monotonic modal logic. Completeness holds for {\sf M} with respect to monotonic modal neighborhood frames~\cite{Hansen}. The logic {\sf M$_{{\sf Nec}}$} extends {\sf M} by adding the axiom $ \square \top $, or equivalently, the necessitation rule. $ {\sf M_{Nec}} $ is complete for monotonic neighborhood frames containing the unit.

The sequent calculus system ${\sf G3M_{Nec}}$ for the classical modal logic ${\sf M_{Nec}}$ was introduced by E. Orlandelli in \cite{Eu} (see Table \ref{G3MN}).
The sequent calculus system ${\sf G3M_{Nec}}$ satisfies several important properties, including: all sequents of the form $ A, \Gamma\Rightarrow \Delta, A $ are derivable; all logical rules are hp-invertible; weakenings and contractions are hp-admissible; and  the cut rule is admissible. Proofs for these properties can be found in \cite{Eu}.

Similar to the translation used by Do\u{s}en in \cite{Do}, we define the $ \square $-translation   from $ \mathcal{L}_{1}=\lbrace  \wedge, \vee, \rightarrow, \supset, \bot\rbrace $, to $ \mathcal{L}_{\square} $, the language of modal propositional logic. It is given by:
\begin{enumerate}
\item ~ $ p^{\square}= p$;
\item ~ $ \bot^{\square}= \bot$;
\item ~$(A\circ B)^{\square}=A^{\square}\circ B^{\square}  $ ~~for $ \circ \in \lbrace \wedge, \vee, \supset\rbrace $;
\item ~$(A\rightarrow B)^{\square}=\square (A^{\square}\supset B^{\square} ) $.
\end{enumerate}
We define $ \Gamma^{\square}:=\lbrace A^{\square} ~|~A \in \Gamma \rbrace$.

A logic $ L_{\square} $ in the language of modal logic, $ \mathcal{L}_{\square} $, is called a \textit{modal companion} of logic  $L$ if, for all $A$ in $\mathcal{L}$, $\vdash_{L } A$ if and only if  $\vdash_{L_{\square} }A^{\square}$.

\begin{thm}\label{3aa}
${\sf G3WF_{N_{2}}}\vdash \Gamma \Rightarrow\Delta$ if and only if ${\sf G3M_{Nec}}\vdash \Gamma^{\square} \Rightarrow\Delta^{\square}$.
\end{thm}

\begin{proof}
Left to right: the proof is by induction on the derivation of $  \Gamma \Rightarrow\Delta $ in ${\sf G3WF_{N_{2}}}$. If $ \Gamma \Rightarrow\Delta $ is an initial sequent, then, since $ p^{\square}=p $ and $ \bot^{\square}=\bot $,  $ \Gamma^{\square} \Rightarrow\Delta^{\square} $ is also an initial sequent of 
$ {\sf G3M_{Nec}} $. For logical rules other than 
$ (LR_{\rightarrow}) $ and $ (R_{\rightarrow}) $, the induction hypothesis and the fact that 
the $ \square $-translation respects 
$ \wedge, \vee $, and 
$ \supset $ suffice to establish the claim. In the case of 
$ (R_{\rightarrow}) $, we have the following:
$$ \dfrac{A\Rightarrow B}{\Gamma\Rightarrow A\rightarrow B, \Delta^{'}} $$
We need to show that we have  $\Gamma^{\square} \Rightarrow \square (A^{\square}\supset B^{\square}),\Delta^{' \square}   $ in ${\sf G3M_{Nec}}$:
$$\dfrac{\dfrac{\dfrac{ }{A^{\square}\Rightarrow B^{\square}}I\!H}{\Rightarrow  A^{\square}\supset B^{\square}}R_{\supset_{\square}}}{\Gamma^{\square} \Rightarrow \square (A^{\square}\supset B^{\square}),\Delta^{' \square} } R_{N_{\square}} $$
In the case of 
$ (LR_{\rightarrow}) $, we have the following:
$$ \dfrac{C\supset D, A\Rightarrow B}{\Gamma^{'}, C\rightarrow D\Rightarrow A\rightarrow B, \Delta^{'}} $$
We need to show that we have  $\Gamma^{'\square}, \square (C^{\square}\supset D^{\square}) \Rightarrow \square (A^{\square}\supset B^{\square}),\Delta^{' \square}   $ in ${\sf G3M_{Nec}}$:
$$\dfrac{\dfrac{\dfrac{ }{C^{\square}\supset D^{\square} , A^{\square}\Rightarrow B^{\square}}I\!H}{C^{\square}\supset D^{\square}\Rightarrow  A^{\square}\supset B^{\square}}R_{\supset_{\square}}}{\Gamma^{' \square}, \square (C^{\square}\supset D^{\square}) \Rightarrow \square (A^{\square}\supset B^{\square}),\Delta^{' \square} } LR_{M_{\square}}  $$

Right to left:
the proof is by induction on the derivation of  $  \Gamma^{\square} \Rightarrow\Delta^{\square} $ in ${\sf G3M_{Nec}}$. If $  \Gamma^{\square} \Rightarrow\Delta^{\square} $ is an initial sequent, then there exists an atomic variable $p \in \Gamma^{\square} \cap\Delta^{\square}$. Since $p^{\square} = p$ and the translation $ \square $ is injective, we have $p \in \Gamma \cap \Delta$. Thus, $\Gamma \Rightarrow \Delta$ is an instance of $(\text{id})$ in ${\sf G3WF_{N_{2}}}$. 
The case where the rule used in the last step of the derivation of $  \Gamma^{\square} \Rightarrow\Delta^{\square} $ is $(L_{\bot})$ is similar. 
As ${\sf G3WF_{N_{2}}}$ and ${\sf G3M_{Nec}}$ share the same logical rules and the $ \square $-translation respects logical operators, the induction hypothesis along with an application of the same rule in ${\sf G3WF_{N_{2}}}$ yields the desired conclusion.
The only remaining cases are $( LR_{M_{\square} })$ and $ ( R_{N_{\square} }) $.

Let $\Gamma^{\square} \Rightarrow \Delta^{\square}$ is the conclusion of an instance of $( LR_{M_{\square}})$. Then, in ${\sf G3M_{Nec}}$, we have:  
$$\dfrac{A \Rightarrow B}{\Sigma, \square A \Rightarrow \square B, \Theta}  LR_{M_{\square} }$$  
That is, $\Gamma^{\square} \equiv \Sigma, \square A$ and $\Delta^{\square} \equiv \square B, \Theta$. 
By the definition of the $\square$-translation, the only formulas that can be translated into $\square$-formulas are strict implications. Therefore, $\Gamma$ must contain a formula $C \rightarrow D$ that is translated into $\square A$, where $A = C^{\square} \supset D^{\square}$. Similarly, $\Delta$ must contain a formula $E \rightarrow F$ that is translated into $\square B$, where $B = E^{\square} \supset F^{\square}$.  
Thus, we have ${\sf G3M_{Nec}} \vdash C^{\square} \supset D^{\square} \Rightarrow E^{\square} \supset F^{\square}$. By the hp-invertibility of $R_{\supset_{\square}}$ in ${\sf G3M_{Nec}}$, we conclude ${\sf G3M_{Nec}} \vdash C^{\square} \supset D^{\square}, E^{\square} \Rightarrow F^{\square}$.  
Let $\Gamma'$ and $\Delta'$ be multisets of formulas that are translated into $\Sigma$ and $\Theta$, respectively. Finally, we have:  
$$\dfrac{\dfrac{}{C \supset D, E \Rightarrow F} I\!H}{\Gamma', C \rightarrow D \Rightarrow E \rightarrow F, \Delta'} LR\rightarrow$$  
showing that ${\sf G3WF_{N_2}} \vdash \Gamma \Rightarrow \Delta$.
\end{proof}

\begin{thm}
Let $ \Gamma $ and $ \Delta $ be finite subset of {\sf Frm}.
$$ \vdash_{{\sf WF_{N_2}}} \bigwedge\Gamma\rightarrow\bigvee\Delta~ ~\rm{ iff}~ ~ \vdash_{{\sf M_{Nec}}} (\bigwedge\Gamma\rightarrow\bigvee\Delta)^{\square} $$
\end{thm}
\begin{proof}
Left to right:
if \( \vdash_{{\sf WF_{N_2}}} \bigwedge\Gamma \rightarrow \bigvee\Delta \), then by Theorems~\ref{3a} and~\ref{3aa}, we have 
${\sf G3M_{Nec}} \vdash \Gamma^{\square} \Rightarrow \Delta^{\square}.$
Since \( {\sf G3M_{Nec}} \) is complete with respect to \( {\sf M_{Nec}} \), it follows that 
$\vdash_{{\sf M_{Nec}}} \bigwedge\Gamma^{\square} \supset \bigvee\Delta^{\square}.$
By applying the necessitation rule, we conclude
$\vdash_{{\sf M_{Nec}}} \square(\bigwedge\Gamma^{\square} \supset \bigvee\Delta^{\square}),$
which is equivalent to 
$\vdash_{{\sf M_{Nec}}}(\bigwedge\Gamma \rightarrow \bigvee\Delta)^{\square}.$

Right to left:
let \( \vdash_{{\sf M_{Nec}}} (\bigwedge\Gamma \rightarrow \bigvee\Delta)^{\square} \). This implies that 
$\vdash_{{\sf M_{Nec}}} \square(\bigwedge\Gamma^{\square} \supset \bigvee\Delta^{\square}),$
and hence ${\sf G3M_{Nec}} \vdash \Rightarrow \square(\bigwedge\Gamma^{\square} \supset \bigvee\Delta^{\square}).$
Given that \( {\sf G3M_{Nec}} \) is cut-admissible and the only rule with a principal boxed formula on the right-hand side is \( (R_{N_{\square}}) \), it follows that ${\sf G3M_{Nec}} \vdash \Rightarrow (\bigwedge\Gamma^{\square} \supset \bigvee\Delta^{\square}).$
By the invertibility of logical rules in \( {\sf G3M_{Nec}} \), we obtain
${\sf G3M_{Nec}} \vdash \Gamma^{\square} \Rightarrow \Delta^{\square},$
which leads to 
${\sf G3WF_{N_2}} \vdash \Gamma \Rightarrow \Delta.$
Thus, 
$\vdash_{{\sf WF_{N_2}}} \bigwedge\Gamma \rightarrow \bigvee\Delta.$
\end{proof}

\section{The system $ {\sf GWF^{s}_{N_{2}}} $ }\label{some}

\begin{table}[h]
\caption{The sequent calculus ${\sf  GWF^{s}_{N_{2}}} $}\label{wfn2s}
\begin{center}
\scalebox{0.85}{
\begin{tabular}{cc}
\hline\hline
\noalign{\smallskip}
\textbf{Initial sequents:} & \\[5pt]

\AxiomC{}
\RightLabel{\(id^{s}\)}
\UnaryInfC{$p, \Gamma \Rightarrow  p$}
\DisplayProof
&
\AxiomC{}
\RightLabel{\(L_{\bot}^{s}\)}
\UnaryInfC{$\bot, \Gamma \Rightarrow Z$}
\DisplayProof\\[10pt]

\multicolumn{2}{l}{ \( p \) is an atomic variable.} \\[10pt]
\noalign{\smallskip}\hline\noalign{\smallskip}
\textbf{logical rules:} & \\[5pt]
\AxiomC{$X, Y, \Gamma \Rightarrow Z$}
\RightLabel{\(L^{s}_{\wedge}\)}
\UnaryInfC{$X \wedge Y, \Gamma \Rightarrow Z$}
\DisplayProof
& 
\AxiomC{$\Gamma \Rightarrow   X$}
\AxiomC{$\Gamma \Rightarrow   Y$}
\RightLabel{\(R^{s}_{\wedge}\)}
\BinaryInfC{$\Gamma \Rightarrow  X \wedge Y$}
\DisplayProof \\[10pt]

\AxiomC{$X, \Gamma \Rightarrow Z$}
\AxiomC{$Y, \Gamma \Rightarrow Z$}
\RightLabel{\(L^{s}_{\lor}\)}
\BinaryInfC{$X \lor Y, \Gamma \Rightarrow Z$}
\DisplayProof\\[10pt]

\AxiomC{$\Gamma \Rightarrow X$}
\RightLabel{\(R^{s}_{\vee_{l}}\)}
\UnaryInfC{$\Gamma \Rightarrow  X \vee Y$}
\DisplayProof
&
\AxiomC{$\Gamma \Rightarrow Y$}
\RightLabel{\(R^{s}_{\vee_{r}}\)}
\UnaryInfC{$\Gamma \Rightarrow X \vee Y$}
\DisplayProof\\[10pt]

\AxiomC{$A\supset B, \Gamma \Rightarrow A$}
\AxiomC{$B, \Gamma \Rightarrow Z$}
\RightLabel{\(L^{s}_{\supset}\)}
\BinaryInfC{$A \supset B, \Gamma \Rightarrow Z$}
\DisplayProof
&
\AxiomC{$A, \Gamma \Rightarrow B$}
\RightLabel{\(R^{s}_{\supset}\)}
\UnaryInfC{$\Gamma \Rightarrow A \supset B$}
\DisplayProof \\[10pt]

\AxiomC{$C\supset D, A \Rightarrow  B$}
\RightLabel{\(LR^{s}_{\to} \)}
\UnaryInfC{$ \Gamma , C\rightarrow D\Rightarrow  A \rightarrow B$}
\DisplayProof
&
\AxiomC{$A \Rightarrow  B$}
\RightLabel{\(R^{s}_{\to} \)}
\UnaryInfC{$\Gamma \Rightarrow A \rightarrow B$}
\DisplayProof\\[14pt]

\multicolumn{2}{l}{ \( A \), \( B \), \( C \), \( D \) \( \in \) \( \sf Frm \) and \( X \), \( Y \), \( Z \) \( \in \)\( \sf Frm_2 \).  } \\[10pt]

\noalign{\smallskip}\hline\hline
\end{tabular}}
\end{center}
\end{table}

Similar to the work of Aboolian and Alizadeh in \cite{ab}, where they introduced a single-succedent variant of \( \sf{GF} \), in this paper, we aim to extend this approach to the logic \( \sf{GWF_{N_2}} \). This extension is for theoretical interest.

\begin{dfn}
The rules of the calculus ${\sf  GWF^{s}_{N_{2}}}$ for the basic subintuitionistic logic ${\sf  WF_{N_{2}}}$ reported in Table \ref{wfn2s}.
\end{dfn}
\begin{lem}\label{aa}
The sequent $ X, \Gamma\Rightarrow  X $ is derivable in ${\sf  GWF^{s}_{N_{2}}}$  for an arbitrary formula $ X $ and arbitrary side formula $ \Gamma $. 
\end{lem}
\begin{proof}
The proof is similar  to the proof of Lemma \ref{1}.
\end{proof}

\begin{lem}\label{bb}
The left rule of weakening is hp-admissible in ${\sf  GWF^{s}_{N_{2}}}$.
$$ \frac{\Gamma\Rightarrow A}{D, \Gamma\Rightarrow A}L^{s}_{w}$$
\end{lem}
\begin{proof}
Left weakening is inherently ensured by allowing arbitrary side formulas in the antecedent of rule conclusions.
\end{proof}

\begin{lem}
The rules $ (L^{s}_{\wedge}) $, $(R^{s}_{\wedge})$, $(L^{s}_{\vee})$and $(R^{s}_{\supset})$ are hp-invertible in ${\sf  GWF^{s}_{N_{2}}}$. Moreover, the rule $(L^{s}_{\supset})$ is
hp-invertible w.r.t. second premise.
\end{lem}
\begin{proof}
The proof is similar to the proof of Lemma 5.4 in \cite{ab}.
\end{proof}

\begin{lem}\label{1c}
The left rule of contraction is hp-admissible in ${\sf  GWF^{s}_{N_{2}}}$.
$$\frac{D, D, \Gamma\Rightarrow A}{D, \Gamma\Rightarrow A}L^{s}_{c}$$
\end{lem}
\begin{proof}
We prove this by induction on the height of the derivation of the premise, focusing specifically on the case where the final rule applied in the derivation is $(LR_{\rightarrow}^{s})$. If one occurrence of $D$ is not the principal formula in the derivation of $D, D, \Gamma \Rightarrow A$, we apply $(LR_{\rightarrow}^{s})$ to the premise of the last rule, where the side formula contains one fewer occurrence of $D$, to derive $D, \Gamma \Rightarrow A$. If neither of the occurrences of $D$ is principal, we apply $(LR_{\rightarrow}^{s})$ in the same way as before, ensuring that the side formula contains one fewer occurrence of $D$, to obtain $D, \Gamma \Rightarrow A$.
\end{proof}
\begin{thm}\label{1d}
The rule of cut  is addmissible in ${\sf  GWF^{s}_{N_{2}}}$.
$$ \frac{\Gamma\Rightarrow  D~~~~~~D, \Gamma^{'}\Rightarrow A}{\Gamma,\Gamma^{'} \Rightarrow A}Cut^{s}$$
\end{thm}
\begin{proof}
The proof is similar to the proof of Lemma 5.6 in \cite{ab}.
\end{proof}
\begin{cor}
If ${\sf  GWF^{s}_{N_{2}}}\vdash \Rightarrow X\vee Y  $, then either $ {\sf  GWF^{s}_{N_{2}}}\vdash \Rightarrow X $ or $ {\sf  GWF^{s}_{N_{2}}}\vdash \Rightarrow Y $.
\end{cor}

\begin{thm}
Let $ \Gamma $ be a multisubset of $ {\sf Frm_{2} }$ and $ X $ be a ${\sf  Frm_{2} }$. If ${\sf  GWF^{s}_{N_{2}}}\vdash_{n} \Gamma \Rightarrow X  $, then ${\sf  GWF_{N_{2}}}\vdash_{n} \Gamma \Rightarrow X  $.
\end{thm}
\begin{proof}
The proof is by induction on the height of derivation of ${\sf  GWF^{s}_{N_{2}}}\vdash \Gamma \Rightarrow X  $ and similar to the proof of Theorem 5.8 in  \cite{ab}.
\end{proof}

\begin{thm}
Let $ \Gamma $ be a finite multisubset of $ {\sf Frm_{2} }$ and $ \Delta $  a  finite multiset of $ {\sf Frm }$. If ${\sf  GWF_{N_{2}}}\vdash \Gamma \Rightarrow \Delta  $, then ${\sf  GWF^{s}_{N_{2}}}\vdash \Gamma \Rightarrow \bigvee\Delta$.
\end{thm}
\begin{proof}
The proof is by induction on the height of the derivation of $\Gamma \Rightarrow \Delta  $ and similar to the proof of Theorem 5.9 in  \cite{ab}.
\end{proof}

\begin{cor}
Let $ \Gamma $ be a finite subset of $ {\sf Frm }$.

$$ \vdash_{{\sf  WF_{N_{2}}}}  \bigwedge \Gamma \rightarrow A ~~~\Leftrightarrow~~~{\sf  GWF_{N_{2}}}\vdash \Gamma \Rightarrow A ~~~\Leftrightarrow~~~{\sf  GWF^{s}_{N_{2}}}\vdash \Gamma \Rightarrow A .$$
\end{cor}
\begin{cor}
If ~$ \vdash_{{\sf  WF_{N_{2}}}} A\vee B  $ ~ then~ $ \vdash_{{\sf  WF_{N_{2}}}} A $~ or~  $ \vdash_{{\sf  WF_{N_{2}}}}  B  $.
\end{cor}
 
\section{Conclusion}
In this paper, a cut-free G3-style sequent calculus ${\sf GWF_{N_{2}}}$ for the subintuitionistic logic ${\sf WF_{N_{2}}}$, along with its single-succedent variant ${\sf GWF^{s}_{N_{2}}}$, is introduced. The initial goal of introducing such a system was to prove the Lyndon's and Craig's interpolation properties for the subintuitionistic logic ${\sf WF_{N_{2}}} $, similar to the approach used by Aboolian and Alizadeh in~\cite{ab}. However, as our work progressed, we realized that we could not apply the same method to prove the interpolation property for ${\sf WF_{N_{2}}} $, and that we need to explore different methods or systems to establish this property. Therefore, one of our main objectives in future work will be to investigate whether we can prove the interpolation property for the subintuitionistic logic ${\sf WF_{N_{2}}}$.

\label{references}
\

\end{document}